\documentclass[11pt]{article}

\usepackage{hyperref}
\usepackage{enumerate, amsthm}
\usepackage{mathrsfs}
\usepackage{rotating}
\usepackage{cancel}
\usepackage{comment}
\usepackage{xcolor}
\usepackage[margin= 0.8 in]{geometry}
\usepackage[margin={1.5cm, 1.5cm},font=small, labelsep=endash]{caption}
\usepackage{mathtools}
\usepackage{tikz}
\usetikzlibrary{patterns}
\usepackage{tikz-dimline}

      \usepackage{amssymb}
      \usepackage{amsmath}
      
      \numberwithin{equation}{section}
      \theoremstyle{plain}
      \newtheorem{theorem}{Theorem}[section]
      \newtheorem{lemma}[theorem]{Lemma}

      \theoremstyle{definition}

      \theoremstyle{remark}
      \newtheorem{remark}[theorem]{Remark}

      \newcommand{\E}{\mathbb E}
      \renewcommand{\P}{\mathbb P}

      \newcommand{\cov}{\mbox{Cov}}
      \newcommand{\Z}{\mathbb{Z}}
      \newcommand{\N}{\mathbb{N}}

      \newcommand{\e}{\mathrm{e}}
      \newcommand{\pa}{\mathcal P}

      \makeatletter
      \def\@setcopyright{}
      \def\serieslogo@{}
      \makeatother
\title{First passage percolation on the exponential of two-dimensional branching random walk}
\author{ Jian Ding\thanks{Partially supported by NSF grant DMS-1455049 and Alfred Sloan fellowship.} \\ University of Pennsylvania \and Subhajit Goswami\footnotemark[1]  \\
Institut des Hautes \'{E}tudes Scientifiques
}

\begin{document}

\maketitle

\begin{abstract}
We consider the branching random walk $\{\mathcal R^N_z: z\in V_N\}$ with Gaussian increments indexed over a two-dimensional box $V_N$ of side length $N$, and we study the first passage percolation where each vertex is assigned weight $e^{\gamma \mathcal R^N_z}$ for $\gamma>0$. We show that for $\gamma>0$ sufficiently small but fixed, the expected FPP distance between the left and right boundaries is at most $O(N^{1 -  \gamma^2/10})$.\newline

\end{abstract}
\section{Introduction}
For $n\in \mathbb N$ and $N = 2^n$, let $V_N \subseteq \mathbb Z^2$ be a box of side
length $N$ whose lower left corner is located at the origin. Let $\mathbb{BD}_j$ denote the 
collection of boxes of the form $([0, 2^{j}-1] \cap \mathbb{Z})^2 + 
(i_1 2^j, i_2 2^j)$ where $i_1, i_2$ and $j$ are nonnegative integers. For $z \in \mathbb{Z}^2$, 
let $\mathbb{BD}_j(z)$ be the (unique) box in $\mathbb{BD}_j$ that contains $v$. Let 
$\{a_{B}\}_{k\geq 0, B\in \mathbb{BD}_k}$ be a collection of i.i.d.\ standard Gaussian variables. 
We define a \emph{branching random walk} $\{\mathcal R_z^N: z \in V_N\}$ with Gaussian increments by
\begin{equation}\label{eq-def-BRW}
     \mathcal{R}_z^N = \sum\limits_{k = 0}^{n-1} a_{\mathbb{BD}_k(z)}.
\end{equation}
For any rectangle $V\subseteq \mathbb Z^2$, we denote by $\Pi_{LR}^V$ the collection of all paths  $\pi$ in $V$ (considered as a subgraph of $\mathbb{Z}^2$) that connect 
the left and right boundaries of $V$. We refer to such a path as a (left-right) \emph{crossing} of $V$. For a 
fixed $\gamma > 0$ we define 
$$D_{\gamma, LR}(V_N) = \min_{\pi \in \Pi_{LR}^{V_N}}\sum\limits_{z \in \pi} \mathrm{e}^{\gamma \mathcal{R}_z^N},$$
to be the first passage percolation distance between the two boundaries of $V_N$ where we assign each vertex $z$ a weight of $e^{\gamma \mathcal R_z^N}$.
\begin{theorem}\label{thm-main}
For a fixed $\gamma_0$ that is small, and all $\gamma < \gamma_0$, we have that
$$\E D_{\gamma, LR}(V_N) = O( N^{1 - \gamma^2/10})\,.$$
\end{theorem}

\begin{remark}
Our proof in fact gives an upper bound on the expected weight for the geodesic crossing through a rectangle. As a result, one can show that the expected weight for the geodesic connecting two fixed vertices has exponent strictly less than 1, by constructing a sequence of ($\log N$ many) rectangles with geometrically growing size that connect these two vertices. 
\end{remark}

\subsection{A brief discussion on motivation and background}

In a broad context, our work is motivated by studying first passage percolation on random media with heavy correlation. The main contribution of our work, is to demonstrate an instance that in a hierarchical random field the exponent of the FPP distance can be strictly less than 1 (despite the fact that the a straight line has weight with exponent strictly larger than 1, and that the majority of the vertices are of values $O(\sqrt{\log N})$). This is, of course, rather different from the classical FPP where the edge/vertex weights are independent and identically distributed. We refer to \cite{ADH15, GK12} and references therein for a review of classical first passage percolation.

In the specific instance of branching random walk, our work is closely related to the first passage percolation problem when the vertex weight is given by exponentiating a two-dimensional discrete Gaussian free 
field (GFF). Inspired by the current work, a similar upper bound on the FPP distance in this case has been recently obtained in 
\cite{DG16}. The FPP on two-dimensional discrete GFF, as we expect, is of fundamental importance in understanding the random metric associated with the Liouville quantum gravity (LQG) \cite{P81, DS11, RV14}. We remark that the random metric of LQG is a major open problem, even ``just'' to make rigorous sense of it. In a recent series of works of Miller and Sheffield, much understanding has been obtained (more on the continuum set up) in the special case of $\gamma = \sqrt{8/3}$; 
see \cite{MS15, MS16} and references therein. Another recent work \cite{GHS16} has provided some bounds on the scaling exponent for a type of LQG metric, though we emphasize that their definition is of very different flavor from the FPP perspective considered in our paper (in particular, no mathematical connection can be drawn between these works at the moment). We refrain ourselves from an extensive discussion on background of LQG in this article. As a final remark to the connection of LQG metric, we note that there are other candidate discrete metrics (associated with LQG or GFF) that have been proposed, and it is possible that these metrics with suitable normalization would give a more desirable 
scaling limit than that of FPP. However, we feel that in the level of precision of the present article and that of \cite{DG16}, it is quite possible that the fundamental mathematical structures (and thus obstacles) are common for these related notions of metrics. 

BRW is perhaps the simplest construction that approximates a log-correlated Gaussian field, of which 
the GFF is a special instance. A number of other properties, especially those regarding to the extreme values of the field, have been proved to exhibit universal behavior among log-correlated Gaussian fields 
(c.f., \cite{Madaule, DRZ15}). Indeed, it is typical that properties were first proved for BRW, and then later for GFF; and it is also typical that the proof in the case of GFF were substantially enlightened by the understanding on BRW. 
However, caution is required when drawing heuristic conclusion on GFF based on BRW for FPP: as demonstrated in \cite{DZ15}, there exists a family of log-correlated Gaussian fields whose exponent for FPP can be made arbitrarily close to 1 (as $K$ grows to $\infty$) if one is allowed to perturb the covariance entry-wise by an additive amount that is at most $K$ (and thus the scaling exponent for FPP is non-universal among log-correlated fields). That being said, we remark that while the mathematical details in \cite{DG16} are substantially more involved and delicate, these two papers do share the same very basic multi-scale analysis proof framework. 

\subsection{Main ideas of our approach}

Our approach is inspired by the rescaling argument employed in the proof of connectivity for the fractal percolation process \cite{CCD88, DM90, Chayes95}, the 
idea of which went back to \cite{ACCFR}. For the purpose of induction, we will in fact work with crossings through 
a rectangle rather than a square. In what follows, we should consider $\gamma$ as a small positive number (but fixed as $N \to \infty$).

We wish to carry out an inductive construction for a light path crossing the rectangle: we will take advantage of the hierarchical structure of BRW and will construct the BRW along the way as we construct light paths in different scales. In order to properly describe our procedure, we need a number of definitions. Let $\Gamma = \Gamma(\gamma) \in [\alpha/\gamma^2, \alpha/\gamma^2 + 2]$ be an odd, positive integer for 
some $\alpha$ to be selected, and let $V_N^{\Gamma} = ([0, \Gamma N - 1] \times [0, N - 1]) 
\cap \Z^2$. The reason for choosing $\Gamma$ odd is mere technical convenience which will 
become apparent later. We will also work under the assumption that $\gamma\ll 1/\alpha \ll 1$ (small or large enough for our bounds or inequalities to hold although we keep this 
implicit in our discussions). We define a Gaussian process $\{\mathcal{R}_z^N: z \in V_{N}^{\Gamma}\}$ on $V_N^{\Gamma}$ as follows:
\begin{itemize}
\item $\{\mathcal{R}_z^N: z \in V_{N;j}^\Gamma\}$ is a BRW on $V_{N;j}^\Gamma = ([(j-1)N, jN - 1] \times [0, N-1]) \cap \Z^2$ where $j \in [\Gamma] \coloneqq \{1, \ldots, \Gamma\}$.
\item $\{\mathcal{R}_z^N: z \in V_{N;j}^\Gamma\}$'s are independent of each other for $j\in [\Gamma]$.
\end{itemize}
So $\{\mathcal{R}_z^N: z \in V_{N}^{\Gamma}\}$ is basically a concatenation of $\Gamma$ 
independent BRWs placed side by side. We can view $V_{N}^{\Gamma}$ as a $\Gamma N \times N$ 
rectangle divided into $\Gamma$ cells of dimension $N \times N$. Similarly, $V_{2N}^{\Gamma}$ can 
be viewed as a $2\Gamma N \times 2N$ rectangle divided into 4 sub-rectangles of dimension $\Gamma N 
\times N$ each of which is a copy of $V_{N}^{\Gamma}$ (see Figure~\ref{fig:geodesic_new} below for 
an illustration). 
\begin{figure}[!htb]
 \centering
\begin{tikzpicture}[semithick, scale = 0.9]
\draw [dashed] (0, -1.9) -- (0, 1.9);
\draw [dashed] (-5.9, 0) -- (5.9, 0);

\draw (0.1, 0.1) rectangle (1.9, 1.9);
\node [scale = 1] at (1, 1) {$a_{B_{n; 1, 2, 1}}$};
\draw (2.1, 0.1) rectangle (3.9, 1.9);
\node [scale = 1] at (3, 1) {$a_{B_{n; 1, 2, 2}}$};
\draw (4.1, 0.1) rectangle (5.9, 1.9);
\node [scale = 1] at (5, 1) {$a_{B_{n; 1, 2, 3}}$};

\draw [decorate,decoration={brace,amplitude=10pt, raise=4pt},yshift=0pt](0.1, 1.9) -- (5.9, 1.9); 
\node [scale = 1, above] at (3, 2.5) {$V_{N; 1, 2}^\Gamma$};

\draw (-0.1, 0.1) rectangle (-1.9, 1.9);
\node [scale = 1] at (-1, 1) {$a_{B_{n; 1, 1, 3}}$};
\draw (-2.1, 0.1) rectangle (-3.9, 1.9);
\node [scale = 1] at (-3, 1) {$a_{B_{n; 1, 1, 2}}$};
\draw (-4.1, 0.1) rectangle (-5.9, 1.9);
\node [scale = 1] at (-5, 1) {$a_{B_{n; 1, 1, 1}}$};

\draw [decorate,decoration={brace,amplitude=10pt, raise=4pt},yshift=0pt](-5.9, 1.9) -- (-0.1, 1.9); 
\node [scale = 1, above] at (-3, 2.5) {$V_{N; 1, 1}^\Gamma$};

\draw (0.1, -0.1) rectangle (1.9, -1.9);
\node [scale = 1] at (-1, -1) {$a_{B_{n; 2, 1, 3}}$};
\draw (2.1, -0.1) rectangle (3.9, -1.9);
\node [scale = 1] at (-3, -1) {$a_{B_{n; 2, 1, 2}}$};
\draw (4.1, -0.1) rectangle (5.9, -1.9);
\node [scale = 1] at (-5, -1) {$a_{B_{n; 2, 1, 1}}$};

\draw [decorate,decoration={brace,amplitude=10pt, mirror, raise=4pt},yshift=0pt](-5.9, -1.9) -- (-0.1, -1.9); 
\node [scale = 1, below] at (-3, -2.5) {$V_{N; 2, 1}^\Gamma$};

\draw (-0.1, -0.1) rectangle (-1.9, -1.9);
\node [scale = 1] at (1, -1) {$a_{B_{n; 2, 2, 1}}$};
\draw (-2.1, -0.1) rectangle (-3.9, -1.9);
\node [scale = 1] at (3, -1) {$a_{B_{n; 2, 2, 2}}$};
\draw (-4.1, -0.1) rectangle (-5.9, -1.9);
\node [scale = 1] at (5, -1) {$a_{B_{n; 2, 2, 3}}$};

\draw [decorate,decoration={brace,amplitude=10pt,mirror, raise=4pt},yshift=0pt] (0.1, -1.9) -- (5.9, -1.9); 
\node [scale = 1, below] at (3, -2.5) {$V_{N; 2, 2}^\Gamma$};
\end{tikzpicture}
 \caption{{\bf $(n+1)$-th stage of the recursive construction of BRW on $V_{2N}^\Gamma$}. Here $\Gamma = 3$.}
 \label{fig:geodesic_new}
\end{figure}
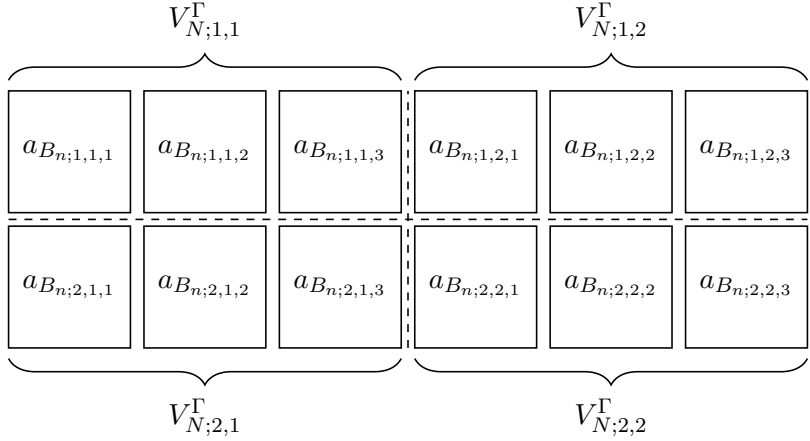

We wish to construct a light path in $V_{2N}^\Gamma$ 
based on constructions in $V_N^\Gamma$. As the path crosses through the rectangle of $V_{2N}^\Gamma$, we will make a choice of whether the path will stay in the upper or lower sub-rectangles based on the Gaussian values associated with the boxes of dimension $N\times N$, and thus we will switch back and forth between top and bottom 
layers. Our goal is to show that the expected weight of the crossing we constructed expands by a factor less than 
2 when the size of the rectangles doubles. Apart from doubling of dimension at each level, there are a number of reasons that the weight will expand: (1) The expected weight assigned to a vertex $z$ i.e. $\E \e^{\gamma 
R_z^N}$ grows by a factor bigger than $1$; (2) We will need to make a construction to connect the paths on the left half and the right half of the rectangle; (3) Every time we switch between top and bottom layer, we will have to add a top-down crossing in the switching location. What we would possibly gain is from the ``variation'' of these Gaussian variables associated with $N\times N$ boxes and thus choosing the favorable layer would reduce 
the weight of our path. 

The aforementioned scheme involves several subtleties. One has to be rather strategic in switchings as there is a cost for that. Eventually, this is reduced to a question of computing the following \emph{regularized total variation} for the Brownian motion, defined by
\begin{equation}\label{eq-regularized-total-variation}
\E \max_k \max_{0 = t_0 < t_1 < t_2 <\ldots < t_k < t_{k+1} = 1} \sum_{i=1}^{k+1} |B_{t_i} - B_{t_{i-1}}| - \lambda k\,.
\end{equation}
Here $B_{[0, 1]}$ is a standard Brownian motion, and $\lambda > 0$ is the term that measures the 
switching cost. We will use a recent result of \cite{Dunlap} on the asymptotic value of 
\eqref{eq-regularized-total-variation} as $\lambda \to 0$.\\

\noindent {\bf Notation convention.} In the rest of the paper we refer to the vertices of $\Z^2$ 
as \emph{points}. For any point $z \in \Z^2$, the horizontal and vertical coordinates of $z$ are 
denoted by $z_x$ and $z_y$ respectively. If $z_y = 0$, we represent $z$ simply as $z_x$, which 
should be clear from the context. A (finite) \emph{path} $\pi$ is a finite sequence $(v_0, v_1, \cdots, v_m)$ of points such that for each $i \geq 0$, $v_i$ and $v_{i + 1}$ are neighbors in $\Z^2$. In each of the pairs of non-negative integers $(n, N), (\ell, 
L), (\ell', L')$ and $(\ell'', L'')$, the second element is always assumed to be 2 raised to the 
power of first element. For functions $F(.)$ and $G(.)$, we write $F = O(G)$ if there exists an 
absolute constant $C > 0$ such that $F \leq C G$ everywhere in the domain. We write $F = 
\Omega(G)$ if $G = O(F)$.

\vspace*{0.2cm}

\noindent {\bf Acknowledgement.} We warmly thank Steve Lalley for numerous helpful discussions and constant encouragement throughout the project. We are grateful to an anonymous referee for a thorough and careful review of an earlier manuscript.

\section{Preliminaries}
\subsection{Coarsening of paths and $\mathcal L$-segments}
The $L$-\emph{coarsening} of a path $\pi = (v_0, 
v_1, \cdots, v_m)$ is defined as follows. Define a sequence of integers $m_0, m_1,\cdots$ recursively as $m_0 = 0$ and,
$$m_j = \inf\{m \geq i \geq m_{j-1} : v_i \notin \mathbb{BD}_\ell(v_{m_{j-1}})\}\mbox{ for }j \geq 1\,.$$
Here we adopt the standard convention that infimum of an 
empty set is $\infty$. Let $p_{L, \pi}$ be the last integer $j$ such that $m_j$ is finite. The $L$-coarsening of $\pi$, denoted as $\pi_{L, coarse}$, can now be defined as the sequence of points $\big(c(\mathbb{BD}_\ell(v_{m_0})), c(\mathbb{BD}_\ell(v_{m_1})), \cdots, 
c(\mathbb{BD}_\ell(v_{m_{p_{L, \pi}}}))\big)$ where 
$c(B)$ is the center of the box $B$. We can also treat $\pi_{L, coarse}$ as a path in $\mathbb{R}^2$ that connects successive points in this list by a line segment. Notice that the coarsening of a simple (i.e. 
self-avoiding) path is not necessarily simple. However if the $L$-coarsening of a (left-right) 
crossing $\pi^*$ of $V_{2L}^\Gamma$ is simple, then the $2L$-coarsening of $\pi^*$ is also simple. 
We will use this fact when we carry out the induction 
step. An \emph{$\mathcal L$-segment at level $\ell$} is a path $\pi$ whose $2^\ell$-coarsening is one of 
two shapes shown in Figure~\ref{fig:L_segment}.
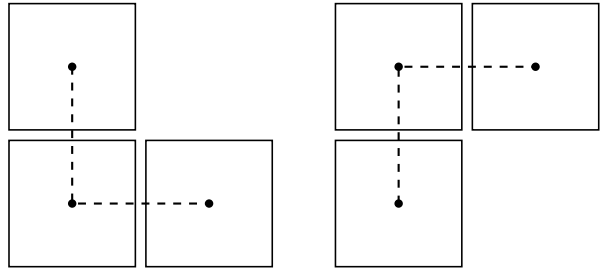
\begin{figure}[!htb]
	\centering
	\begin{tikzpicture}[semithick, scale = 1.4]
	\draw (-0.6, -0.6) rectangle (0.6, 0.6);
	\fill (0, 0) circle [radius = 0.04];
	\draw (0.7, -0.6) rectangle (1.9, 0.6);
	\fill (1.3, 0) circle [radius = 0.04];
	\draw (-0.6, 0.7) rectangle (0.6, 1.9);
	\fill (0, 1.3) circle [radius = 0.04];
	
	\draw [thick, dashed] (0, 1.3) -- (0, 0) -- (1.3, 0);
	
	\draw (2.5, 0.7) rectangle (3.7, 1.9);
	\fill (3.1, 1.3) circle [radius = 0.04];
	\draw (3.8, 0.7) rectangle (5, 1.9);
	\fill (4.4, 1.3) circle [radius = 0.04];
	\draw (2.5, -0.6) rectangle (3.7, 0.6);
	\fill (3.1, 0) circle [radius = 0.04];
	
	\draw [thick, dashed] (3.1, 0) -- (3.1, 1.3) -- (4.4, 1.3);
	\end{tikzpicture}
	\caption{{\bf Two possible $2^\ell$-coarsenings of an $\mathcal L$-segment at level $\ell$}.}
	\label{fig:L_segment}
\end{figure}
\subsection{Switched sign construction for branching random walk}
We first introduce a few more notations. As described in the introduction, we view $V_{2N}^{\Gamma}$ as a $2\Gamma N \times 2N$ rectangle divided into 4 sub-rectangles of dimension $\Gamma N \times N$ each of which is a copy of $V_{N}^{\Gamma}$ (see Figure~\ref{fig:geodesic_new} 
for an illustration). Denote these 4 subrectangles (or the copies of $V_N^{\Gamma}$) as $\{V_{N;i,k}^{\Gamma}\}_{i\in [2], k \in [2]}$ in the 
usual order. We can also shift the origin by $u = (\Gamma i 2^{n+1}, j 2^{n+1})$ for $i, j \in \N \cup \{0\}$ and denote the corresponding subrectangles (respectively rectangles) by $\{V_{N;i,k}^{\Gamma, u}\}_{i\in [2], k 
\in [2]}$ (respectively $V_{2N}^{\Gamma, u}$). Similarly we can define the fields $\{\mathcal{R}_{z; i, k}^{N, u}: z \in V_{N;i,k}^{\Gamma, u}\}$ or $\{\mathcal{R}_{z}^{N, 
u}: z \in V_{N}^{\Gamma, u}\}$. The $j$-th constituent cell in $V_{N;i,k}^{\Gamma, u}$ (which is box of side length $N$) will be denoted as $B_{n;i,k,j}^u$ where $j \in [\Gamma]$. We observe that the Gaussian field $\{\mathcal{R}_{z}^{2N, u}: z \in V_{2N}^{\Gamma, u}\}$ is obtained by adding $a_{B_{n;i,k,j}^u}$ to the field $\{\mathcal{R}_{z; i, k}^{N, u}: z \in V_{N;i,k}^{\Gamma, u}\}$ in the box $B_{n;i,k,j}^u$ for $i \in [2], k \in [2]$ and $j \in 
[\Gamma]$ (See Figure~\ref{fig:geodesic_new}). We will omit the additional superscript $u$ in all these notations when $u = 0$.

We will represent BRW by a construction in the fashion of switching-signs, which is tailored to our inductive 
construction for light crossings. To this end, we denote the collection of all rectangles of the form $([0, 2^\ell - 1] \times [0, 2^{\ell+1} - 1]) \cap \Z^2 + (i2^\ell, j2^\ell)$ by $\mathbb{B}\mathbb{D}_{\ell,2}$ where $i, j, 
\ell$ are nonnegative integers. Like $\mathbb{BD}(z)$, we use $\mathbb{BD}_{\ell, 2}(z)$ to denote the unique rectangle in $\mathbb{BD}_{\ell, 2}$ containing $z \in 
\Z^2$. Let $A_{\ell,\Gamma}$ be the collection of all 
points of the form $(\Gamma i 2^\ell, j 2^\ell)$. Denote by $R_{\ell; k, j}^u$ the rectangle formed by the boxes $B_{\ell;1, k, j}^u$ and $B_{\ell;2, k, j}^u$ when $u \in 
A_{\ell+1, \Gamma}$ (see Figure~\ref{fig:switch_BRW}). Note that $R_{\ell; k, j}^u\in 
\mathbb{B}\mathbb{D}_{\ell, 2}$. Also let $\{a_{B}\}_{B \in \mathbb{B}\mathbb{D}_{k,2}, k \geq 0}$ be a collection of i.i.d.\ standard normal random variables, that is independent of $\{a_{B}\}_{B\in \mathbb{BD}_k, 
k\geq 0}$. We will now construct a family of Gaussian fields $\{\chi_z^{L, u}: z \in V_L^{\Gamma, u}\}$ for $u \in A_{\ell, \Gamma}$ and $\ell =0, 1, 2, \ldots, n$ recursively as follows:
\begin{itemize}
\item $\chi_z^{1, u} = 0$ for all $z \in V_1^{\Gamma, u}$.
\item Let $u \in A_{\ell+1, \Gamma}$ and $\{\chi_{z; i, k}^{L, u}: z \in V_{L;i,k}^{\Gamma, u}\}$ be the field on $V_{L; i,k}^{\Gamma, u}$ constructed in level $\ell$. For $z \in B_{\ell; i, k, j}^u$, 
define 
\begin{equation}
\label{eq:chi_def}
\chi_z^{2L, u} = \chi_{z; i, k}^{L, u} + (-1)^{(k - 1)\Gamma + j}b_\ell a_{\mathbb{BD}_{\ell + 1}(z)} + (-1)^i c_\ell a_{\mathbb{BD}_{\ell, 2}(z)}\,,
\end{equation}
where $b_\ell = \sqrt{(1 - 4^{-\ell})/3}$ and $c_\ell = 
\sqrt{2(1 - 4^{-\ell})/3}$.
\end{itemize}
A few remarks about the above construction might be 
helpful. For $z \in B_{\ell; i, k, j}^u$, the box $\mathbb{BD}_{\ell + 1}(z)$ is simply $V_{2L; \lceil( (k-1)\Gamma + j) / 2 \rceil}^{\Gamma, u} = u + V_{2L; \lceil ((k-1)\Gamma + j) / 2 \rceil}^\Gamma$ whereas 
$\mathbb{BD}_{\ell, 2}(z)$ is $R_{\ell; k, j}^u$. Thus we add the same variable, namely
\begin{equation}
\label{eq:Z_def}
Z_{\ell; i, k, j} = (-1)^{(k - 1)\Gamma + j}b_\ell a_{V_{2L; \lceil((k-1)\Gamma + j) / 2 \rceil}^u} + (-1)^i c_\ell a_{R_{\ell; k, j}^u}\,,
\end{equation}
to $\chi_{z; i, k}^{L, u}$ on $B^u_{\ell; i, k, j}$. 
$Z_{\ell; 1, k, j}$ and $Z_{\ell; 2, k, j}$ differ only 
by the sign of $c_\ell a_{R_{\ell; k, j}^u}$. Also notice that $\mathbb{BD}_{\ell + 1}(z)$ is same on the $2m + 1$-th and $2m + 2$-th rectangles (i.e. the 
$R_{\ell; \cdot, \cdot}^u$'s) in $V_{2L}^{\Gamma, u}$. However the sign of $b_\ell a_{\mathbb{BD}_{\ell + 
1}(z)}$ changes inside $Z_{\ell; i, k, j}$. See Figure~\ref{fig:switch_BRW} for an illustration 
of this construction at level $\ell + 1$ (we drop the superscript $u$ whenever $u = 0$). 

Since $c_\ell^2 = 2b_\ell^2$, the distribution of the field $\{\chi_z^N: z \in V_{N; j}^\Gamma\}$ is symmetric with respect to symmetries of the box $V_{N; j}^\Gamma$. This fact will be used later when we construct crossings through $V_{N;j}^\Gamma$'s in the vertical direction. Finally, we define 
$$\tilde{\chi}_z^{N, u}= b_na_{\mathbb{BD}_{\ell+1; 2}(z)} + \chi_z^{N, u} \mbox{ for }u \in A_{n, \Gamma} \mbox{ and } z \in V_{N; j}^{\Gamma, u} = u + V_{N; j}^\Gamma\,.$$
\begin{lemma}
\label{lem:switch_brw}
The Gaussian fields $\{\tilde{\chi}_z^{N}: z \in V_N^\Gamma\}$ and 
$\{\mathcal R_z^{N}: z \in V_N^{\Gamma}\}$ are identically distributed.
\end{lemma}
\begin{proof}
It is not difficult to check that the fields $\{\tilde{\chi}_z^N: z \in V_{N; j}\}$'s are 
independent and identically distributed for $j \in [\Gamma]$. Hence it suffices to verify that we 
have the correct covariances between field values at all pairs of vertices inside $V_N$. Toward 
this end take a pair $(u, v)$ of vertices in $V_N$ which were separated until level $\ell$,  i.e., 
$\cov(\mathcal R_{u}^N, \mathcal R_{v}^N) = n - \ell$. The covariance between $\chi_u^N$ and 
$\chi_v^N$ is given by
$$\cov(\chi_u^N, \chi_v^N)  = -(1 - 4^{-\ell})/3 + \mbox{$\sum_{\ell + 1 \leq m \leq n}$}(1 - 
4^{-m}) = n - \ell - (1 - 4^{-n})/3\,.$$
Consequently $\cov(\tilde{\chi}_u^N, \tilde{\chi}_v^N) = \cov(\chi_u^N, \chi_v^N) + (1 - 4^{-n})/3 
=  n - \ell$. Hence the lemma follows.
\end{proof}
\begin{remark}
It is clear that it suffices to work with $\{\chi_z^N: z\in V_N^\Gamma\}$ rather than the BRW in 
what follows. Using this construction, in every scale of our optimization, we crucially have effective variance $c_\ell^2 \approx \frac{2}{3}$ that is larger than $1/2$ (which would be the variance we effectively used if we work with the more canonical construction of BRW). 
\end{remark}
\begin{figure}[!htb]
 \centering
\begin{tikzpicture}[semithick, scale = 1]
\draw (-4.7, 0.1) -- (-1.7, 0.1);
\node [scale = 0.8] at (-3.2, 1.6) {$-b_\ell a_{V^\Gamma_{2L;1}} - c_\ell a_{R_{\ell; 1, 1}}$};
\draw (-1.5, 0.1) rectangle (1.5, 3.1);
\node [scale = 0.8] at (0, 1.6) {$b_\ell a_{V^\Gamma_{2L;1}} - c_\ell a_{R_{\ell; 1, 2}}$};
\draw (1.7, 0.1)  rectangle (4.7, 3.1);
\node [scale = 0.8] at (3.2, 1.6) {$-b_\ell a_{V^\Gamma_{2L;2}} - c_\ell a_{R_{\ell; 1, 3}}$};
\draw [decorate,decoration={brace, amplitude = 10pt, raise=4pt},yshift=0pt](-4.7, 3.1) -- (4.7, 3.1); 
\node [scale = 1, above] at (0, 3.6) {$V^\Gamma_{L; 1, 1}$};

\draw (-4.7, -0.1) rectangle (-1.7, -0.1);
\node [scale = 0.8] at (-3.2, -1.6) {$-b_\ell a_{V^\Gamma_{2L;1}} + c_\ell a_{R_{\ell; 1, 1}}$};
\draw (-1.5, -0.1) rectangle  (1.5, -3.1);
\node [scale = 0.8] at (0, -1.6) {$b_\ell a_{V^\Gamma_{2L;1}} + c_\ell a_{R_{\ell; 1, 2}}$};
\draw (1.7, -0.1)  rectangle  (4.7, -3.1);
\node [scale = 0.8] at (3.2, -1.6) {$-b_\ell a_{V^\Gamma_{2L;2}} + c_\ell a_{R_{\ell; 1, 3}}$};
\draw [decorate,decoration={brace, amplitude = 10pt, mirror, raise=4pt},yshift=0pt](-4.7, -3.1) -- (4.7, -3.1); 
\node [scale = 1, below] at (0, -3.6) {$V^\Gamma_{L; 2, 1}$};

\draw [red] (-4.7, -3.1) -- (-4.7, 3.1) -- (-1.7, 3.1) -- (-1.7, -3.1) -- (-4.7, -3.1);

\end{tikzpicture}
 \caption{{\bf Construction of the field $\chi_z^{2L}$ on $V_{L;i,1}^{\Gamma}$}. Here $\Gamma = 
 3$. The red lines indicate the boundary of $R_{\ell; 1, 1}$.}
 \label{fig:switch_BRW}
\end{figure}
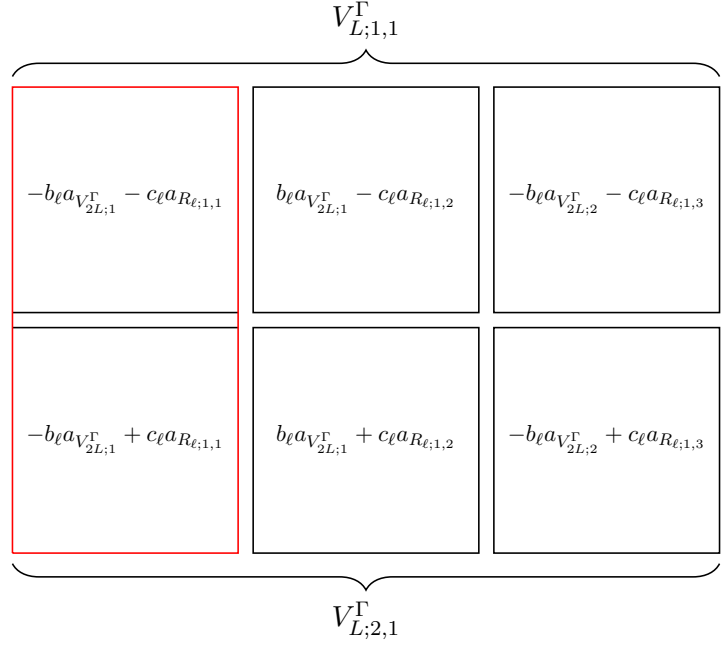

\subsection{Regularized total variation of Brownian motion}
\label{sec:total_variation}
In this subsection we present a result from \cite{Dunlap} that provides a lower bound on the functional defined in \eqref{eq-regularized-total-variation}. This is only one part of the main result in \cite{Dunlap} which also provides an upper bound.

\begin{theorem}
\label{thm:brownian_switch}
Let $\{W_t\}_{0 \leq t \leq 1}$ be a Standard Brownian Motion and $1 > \lambda > 0$. For any partition $P = [t_0, t_1, \cdots, t_{k+1}]$ of $[0, 1]$ where $0 = t_0 < t_1 < \cdots < t_{k+1} = 1$, define the quantity $\Phi_{P, \lambda}$ as
$$\Phi_{P, \lambda} = \mbox{$\sum_{0 \leq i \leq k}$} |W_{t_{i + 1}} - W_{t_i}|  - \lambda k\,.$$
Then there exists a (random) partition $\pa = [t_0^*, t_1^*, \cdots, t_{k^* + 1}^*]$ of $[0, 1]$ such that the following conditions hold:
\begin{enumerate}[(a)]
\item The random vector $(t_1^*, \cdots, t_{k^*}^*, k^*)$ is measurable relative to $\sigma(\{W_t\}_{0 \leq t \leq 1})$.
\item $\E \Phi_{P, \lambda} \geq \frac{1}{\lambda}$.
\item $W_{t_1^*} - W_{t_0^*}, W_{t_2^*} - W_{t_1^*}, \cdots, W_{t_{k^* + 1}^*} - W_{t_{k^*}^*}$ have alternating signs. Furthermore $\P(W_{t_1^*} - W_{t_0^*} > 0) = 0.5$.
\item $|W_{t_{j + 1}^*} - W_{t_j^*}| \geq \lambda$ for all $1 \leq j \leq k^* - 1$.
\end{enumerate}
\end{theorem}
\section{Inductive construction for light crossings}
\subsection{The induction hypotheses}
We will construct a light (left-right) crossing through $V_{L}^\Gamma$ (or equivalently 
$V_{L}^{\Gamma, u}$) inductively for $L = 2^\ell$ and $\ell = 0, \ldots, n$. Throughout this 
section whenever we refer to a crossing of $V_L^{\Gamma, u}$, it is implicitly assumed 
that the underlying field is $\{\chi_z^{L, u}: z \in V_{L}^{\Gamma, u}\}$ unless stated 
otherwise. For technical convenience, let us assume that for each $\ell \geq 0$ and $u \in 
A_{\ell+1, \Gamma}$ we have an independent standard Brownian motion $\{W_t^{u, \ell}\}_{0 \leq t \leq 2\Gamma}$ that is also independent of $\{a_{B}\}_{B \in \mathbb{B}\mathbb{D}_k, k \geq 0}$. We sample $a_{R_{\ell; k, j}^u}$'s as increments of this process at appropriate time points. We are now ready to state our induction hypotheses. For $\ell = 0$, we 
have only one option, i.e., to take the straight line as our crossing. This crossing has weight exactly $\Gamma$. Suppose that for each $1 \leq \ell' \leq \ell$ and $u \in A_{\ell', \Gamma}$, we have a crossing $\pi^{*,u,\ell'}$ of $V_{2^{\ell'}}^{\Gamma, u}$, identically distributed for different $u$, that satisfies the following properties:
\begin{enumerate}
\item \label{induct:hypo2} $\pi^{*,u,\ell'}_{L'/2, coarse}$ is simple and measurable relative to $\sigma\big(\{W_t^{u, \ell' - 1}\}_{0 \leq t \leq 2\Gamma}\big)$.
\item \label{induct:hypo3} The Law of $\pi^{*,u,\ell'}_{L'/2, coarse}$ is invariant with respect 
to reflection about the line $y = u_y + 2^{\ell' - 1} - 0.5$, i.e. the horizontal line passing through the center of $V_{2^{\ell'}}^{\Gamma, u}$.
\item \label{induct:hypo4}For $j \in [\Gamma]$, denote the line $x = u_x + jL' - 1$ as $\mathbb L_{jL'}$. Let $z_{j, \ell'}^u$ be the point in the set $\{u, u + (0, L'/2), $ $u + (\Gamma L'/2, 0), u + (\Gamma L'/2, L'/2)\}$ such that the point where $\pi^{*, u, \ell'}$ hits $\mathbb L_{jL'}$ for the last time lies within $V_{L'/2}^{\Gamma, z_{j, \ell'}^u}$. Also for $j \in \mathbb N^+$, let $v_{j, L', L''}^u$ be the last point along $\pi^{*,u, \ell'}_{L'', coarse}$ 
that lies to the left of $\mathbb L_{jL'}$. Then $v_{j, L', L'_{100}}^u = v_{2j, L'/2, 
L'_{100}}^{z_{j, \ell'}^u}$, i.e., the crossings $\pi^{*, u, \ell'}$ and $\pi^{*, z_{j, \ell'}^u, 
\ell' - 1}$ hit $\mathbb L_{jL'}$ for the last time in the same $L'_{100} \times L'_{100}$ box. Here $L'_{100} = \max(2^{-100} L', 1)$.
\item \label{induct:hypo5}$\E (\sum_{z \in \pi^{*, u, \ell'}}\mathrm{e}^{\gamma \chi_z^{\ell'}}) \leq 2\Gamma(2 - 0.2\gamma^2)^{\ell'}$.
\end{enumerate}
One of the consequences of our induction hypotheses, which is crucial for our analysis, is that 
$v_{j, L, \delta L}^u$ is uniformly distributed for all $\ell' \leq \ell$ as given by the 
following lemma.
\begin{lemma}
\label{lem:uniform}
Let $R_{\ell; j}^u$ be the unique $L_{100} \times L$ rectangle that shares its right boundary with 
$V_{L; j}^u$. Also let $\{B_{\ell; j, m}^u\}_{1 \leq m \leq L/L_{100}}$ be the collection of boxes in $\mathbb{BD}_{\log_2 L_{100}}$ that 
comprise $R_{\ell; j}^u$. Now if the induction hypotheses \ref{induct:hypo2}, \ref{induct:hypo3} and \ref{induct:hypo4} hold for all $\ell' \leq \ell$, then $v_{j, L, L_{100}}^u$ is distributed uniformly on the set $\{c_{\ell; j, m}^u\}_{1 \leq m \leq L / L_{100}}$, where $c_{\ell; j,m}^u$ is the center of the box $B_{\ell; j,m}^u$. 
\end{lemma}
\begin{proof}
Notice that $v_{j, L, L/2}^u$ is determined by whether $\pi_{L/2, coarse}^{*, u, \ell}$ exits the box $V_{L; j}^u$ through the segment $\{u_x + jL - 1\} \times [u_y, u_y + L/2 - 1]$ or through the segment $\{u_x + jL - 1\} 
\times [u_y + L/2, u_y + L - 1]$. Hypothesis~\ref{induct:hypo3} tells us that these choices 
have probability $1/2$ each. Conditional on such a 
choice, we then look at $v_{2j, L/2, L/4}^{z_{j, 
\ell}^u}$ (see hypothesis~\ref{induct:hypo4}). By hypothesis~\ref{induct:hypo2}, $v_{2j, L/2, L/4}^{z_{j, 
\ell}^u}$ is independent of $v_{j, L, L/2}^u$. 
Consequently by hypothesis~\ref{induct:hypo3}, the two possible choices for $v_{2j, L/2, L/4}^{z_{j, \ell}^u}$ 
have probability $1/2$ each given the choice of $v_{j, L, 
L/2}^u$. Now noting that $v_{2j, L/2, 
L/4}^{z_{j, \ell}^u} = v_{j, L, L/4}^u$ by hypothesis~\ref{induct:hypo4}, we get that $v_{j, L, L/4}^u$ is uniformly distributed over the set of all 
4 possible choices. Iterating this argument $\log_2 L/L_{100}$ times gives us the lemma.
\end{proof}
\subsection{The induction step}
Now we will carry out the induction step. It suffices to 
produce a crossing for $u = 0$. As usual, we will drop 
$u$ from all the superscripts when $u$ is origin. In the following paragraph we give a broad overview of our 
strategy. This will be explained in more detail later, as we move to the different cases.

We will first decide the $L$-coarsening of $\pi^{*, \ell 
+ 1}$. It would consist of a sequence $H_1, H_2, \cdots, H_{R + 1}$ of horizontal segments (possibly of length 0) from left to right such that every two successive 
segments are connected by a vertical segment. See Figure~\ref{fig:coarse_construct} for an illustration. 
Notice that this is the only possible option for $\pi^{*, \ell + 1}_{L, coarse}$, if it has to be simple as required by induction hypothesis~\ref{induct:hypo2}. 
So we can ``encode'' $\pi^{*, \ell + 1}_{L, coarse}$ as a sequence of $\{1, 2\}$ valued random variables $\{i_{j, k}\}_{j \in [\Gamma], k \in [2]}$, where $i_{k, j} = 1$ or $2$ accordingly as $\pi^{*,\ell + 1}_{L, coarse}$ enters the rectangle $R_{\ell; k, j}$ through $B_{\ell; 1, k, j}$ or $B_{\ell; 2, k, j}$ respectively. 
We refer to this as a \emph{switching strategy}. After defining $\pi^{*, \ell + 1}_{L, coarse}$, we will 
construct the path $\pi^{*, \ell + 1}$. From the previous level we have four crossings $\{\pi^{*, \ell}_{i, k}\}_{i \in [2], k \in [2]}$ respectively in $\{V_{L;i,k}^\Gamma\}_{i \in 
[2], k \in [2]}$. We will ``join'' the paths $\pi^{*, \ell}_{i, 1}$ and $\pi^{*, \ell}_{i, 2}$ into a crossing 
$\pi^{*, \ell+1}_i$ of $V_{2L}^\Gamma$ for $i \in [2]$. 
Now let $H_1^* = H_1$ and for $2 \leq r \leq R + 1$, let $H_r^*$ be the subsegment of $H_r$ consisting of all but 
the left endpoint of $H_r$. Thus the points in $\cup_{r \in [R+1]} H_r^*$ span the entire horizontal range of points in $\cup_{r \in [R+1]} H_r$ while no two points in the former share the same horizontal coordinate. 
Since $L$-coarsenings of the crossings at level $\ell$ are also simple, it follows that the vertices of $\pi_i^{*, \ell + 1}$ lying within the $L \times L$ boxes intersecting $H_r^\star$ define a subpath of 
$\pi_i^{*, \ell + 1}$, say, $h_r^{*, \ell + 1}$. Here $i$ is 1 or 2 accordingly as $H_r$ lies in the top or bottom row of $V_{2L}^\Gamma$ (See Figure~\ref{fig:coarse_construct} and the discussions 
preceding Figure~\ref{fig:geodesic_new}). Finally we will connect $h_r^{*, \ell + 1}$ and $h_{r+1}^{*, \ell + 1}$ by an appropriate $\mathcal L$-segment $v_r^{*, \ell 
+ 1}$ at level $\ell$. The paths $h_1^{*, \ell + 1}, v_1^{*, \ell + 1}, h_2^{*, \ell + 1}, \cdots, v_{R}^{*, \ell + 1}, h_{R + 1}^{*, \ell + 1}$ define the crossing 
$\pi^{*, \ell + 1}$. Note that the $L$-coarsening of $\pi^{*, \ell + 1}$ is the very one that we started with.

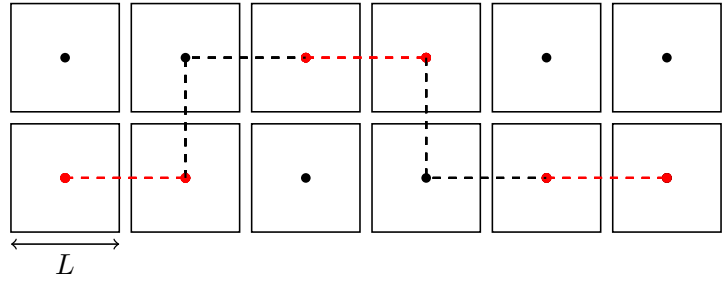
\begin{figure}[!htb]
 \centering
\begin{tikzpicture}[semithick, scale = 0.8]
\draw (0.1, 0.1) rectangle (1.9, 1.9);

\draw (2.1, 0.1) rectangle (3.9, 1.9);

\draw (4.1, 0.1) rectangle (5.9, 1.9);

\draw (-0.1, 0.1) rectangle (-1.9, 1.9);

\draw (-2.1, 0.1) rectangle (-3.9, 1.9);

\draw (-4.1, 0.1) rectangle (-5.9, 1.9);

\draw (0.1, -0.1) rectangle (1.9, -1.9);

\draw (2.1, -0.1) rectangle (3.9, -1.9);

\draw (4.1, -0.1) rectangle (5.9, -1.9);

\draw (-0.1, -0.1) rectangle (-1.9, -1.9);

\draw (-2.1, -0.1) rectangle (-3.9, -1.9); 

\draw (-4.1, -0.1) rectangle (-5.9, -1.9);
\draw [<->] (-5.9, -2.1) -- (-4.1, -2.1);
\node [scale = 1, below] at (-5, -2.1) {$L$};

\foreach \x in {-5, -3, -1, 1, 3, 5}{

\foreach \y in {-1, 1}{
         \fill (\x, \y) circle [radius = 0.08];
}

\fill[red] (-5, -1) circle [radius = 0.08];
\fill[red] (-3, -1) circle [radius = 0.08];
\fill[red] (-1, 1) circle [radius = 0.08];
\fill[red] (1, 1) circle [radius = 0.08];
\fill[red] (3, -1) circle [radius = 0.08];
\fill[red] (5, -1) circle [radius = 0.08];

\draw [thick, dashed, red] (-5, -1) -- (-3, -1);

\draw [thick, dashed] (-3, -1) -- (-3, 1) -- (-1, 1);

\draw [thick, dashed, red] (-1, 1) -- (1, 1);

\draw [thick, dashed] (1, 1) -- (1, -1) -- (3, -1);

\draw [thick, dashed, red]  (3, -1) -- (5, -1);
}
\end{tikzpicture}
 \caption{{\bf $L$-coarsening of $\pi^{*, \ell+1}$}. Here $\Gamma = 3$ and $R = 2$. The three red segments are $h_1^{*, \ell + 1}, h_2^{*, \ell + 1}$ and $h_3^{*, \ell + 1}$ from left to right.}
 \label{fig:coarse_construct}
\end{figure}
We also need some notations in order to track the change in expected weight of crossings between 
two levels. To this end define,
\begin{align*}
D_{\gamma, \ell} &= \sum_{v \in \pi^{*, \ell}} e^{\gamma \chi_v^L} \mbox{ and } d_{\gamma, \ell} 
= \E  D_{\gamma, \ell}\,,\\
D_{\gamma, \ell, j} &=  \sum_{v\in \pi^{*,\ell} \cap V_{L, j} } e^{\gamma \chi_v^L} \mbox{ and } 
d_{\gamma, \ell, j}  = \E D_{\gamma, \ell, j} \mbox{ for } j\in [ \Gamma]\,.
\end{align*} 
Also define for $k \in [2]$,
$$\tilde{D}_{\gamma, \ell; k} = \sum_{j \in [\Gamma]}\sum_{v\in \pi_{i_{j, k}, k}^{*, \ell}} e^{\gamma \chi_v^{2L}}\,,$$
where $\{i_{j, k}\}_{j \in [\Gamma], k \in [2]}$ is the 
switching strategy that we use. Notice that if $\pi^{*, \ell + 1}$ is constructed according to the general strategy described in the last paragraph, then $\tilde{D}_{\gamma, \ell; k}$ gives an upper bound on the total weight of vertices in $\pi^{*, \ell + 1}$ coming from $\pi_{i, k}^{*, \ell}$'s. 
Recalling from \eqref{eq:chi_def} and \eqref{eq:Z_def} that $\chi_z^{2L} = \chi_{z; i, k}^{L} + Z_{\ell; i, k, j}$ for all $z \in B_{\ell;i, k, j}$, we can then write
$$\E \tilde{D}_{\gamma, \ell; k} = \sum_{j \in [\Gamma]}\E \Big(\e^{\gamma Z_{\ell; i, k, j}}\sum_{v\in \pi_{i_{j, k}, k}^{*, \ell}} e^{\gamma \chi_{v; i, k}^{L}}\Big) = \sum_{j \in [\Gamma]}\E \e^{\gamma Z_{\ell; i, k, j}}\E \Big(\sum_{v\in \pi_{1, k}^{*, \ell}} e^{\gamma \chi_{v; 1, k}^{L}}\Big) = \sum_{j \in [\Gamma]} d_{\gamma, \ell, j}\e^{\gamma Z_{\ell; i, k, j}}\,,$$
where for the second step we assumed that the switching strategy is independent of $\{\chi_{.}^{L', .}\}$ for 
$L' \leq L$ (induction hypothesis~\ref{induct:hypo2}). 
Writing 
$$\e^{\gamma Z_{\ldots}} = (1 + \gamma Z_{\ldots} + \frac{a_{\ell + 1}^2\gamma^2}{2}) + \frac{\gamma^2}{2}(Z_{\ldots}^2 - a_{\ell + 1}^2) + \big(\e^{\gamma Z_{\ldots}} - (1 + \gamma Z_{\ldots} + \frac{\gamma^2Z_{\ldots}^2}{2})\big)\,$$
for $a_{\ell+1}^2 = b_{\ell+1}^2 + c_{\ell+1}^2 = 1 - 4^{-(\ell + 1)}$, we get
\begin{equation}
\label{eq:weight_upper_bnd}
\E \tilde{D}_{\gamma, \ell; k} = \E \sum_{j = 1}^\Gamma d_{\gamma, \ell, j} \big(1 + \gamma 
Z_{\ell; i_{j, k}, k, j} + \frac{a_{\ell + 1}^2\gamma^2}{2}\big) + \E \mathrm{Err}_{\gamma, \ell; k, 1} + \E \mathrm{Err}_{\gamma, \ell; k, 2}\,,
\end{equation}
where, $\mathrm{Err}_{\gamma, \ell; k, 1} = \frac{\gamma^2}{2}\sum_{j = 1}^\Gamma d_{\gamma, \ell, j} (Z_{\ell; i_{j, k}, k, 
j}^2 - a_{\ell + 1}^2)$ and
$$\mathrm{Err}_{\gamma, \ell; k, 2} = \sum_{j = 1}^\Gamma d_{\gamma, \ell, j} \big(\mathrm{e}^{\gamma Z_{\ell; i_{j, k}, k, j}} - (1 + \gamma Z_{\ell; i_{j, k}, k, j} + \frac{\gamma^2Z_{\ell; i_{j, k}, k, j}^2}{2})\big)\,.$$
It is easy to see that $\E |\mathrm{Err}_{\gamma ,\ell; k, 2}| = d_{\gamma, \ell}O(\gamma^3)$. On 
the other hand, we get the following from expanding the terms in $\mathrm{Err}_{\gamma, \ell; k, 
1}$ (see \eqref{eq:Z_def}):
\begin{align*}
\mathrm{Err}_{\gamma, \ell; k, 1} = \frac{b_\ell^2\gamma^2}{2}\sum_{j = 1}^\Gamma d_{\gamma, \ell, j} (a_{V_{2L; \lceil( (k-1) \Gamma + j)/2\rceil }}^2 - 1) + \frac{c_{\ell}^2\gamma^2}{2}\sum_{j = 1}^\Gamma d_{\gamma, \ell, j}(a_{R_{\ell; k, j}}^2 - 1) \\ 
+ \gamma^2 b_\ell c_\ell (-1)^{(k-1)\Gamma}\sum_{j = 1}^\Gamma d_{\gamma, \ell, j} 
a_{V_{2L; \lceil( (k-1) \Gamma + j)/2\rceil }} (-1)^{j + i_{j, k}}a_{R_{\ell; k, j}}\,.
\end{align*}
The expectations of the first two terms in the expansion above are obviously 0. As to the third term, 
recall that our switching mechanism is independent of the random variables $a_{V_{2L; \lceil( (k-1) \Gamma + j)/2\rceil }}$'s by induction hypothesis~\ref{induct:hypo2}. Consequently $\E \mathrm{Err}_{\gamma, \ell; k, 
1} = 0$. Thus we only need to reckon with the first term in the right hand side of 
\eqref{eq:weight_upper_bnd} so far as the contributions from $\pi_{i, k}^{*, \ell}$'s are 
concerned. This will, of course, depend on the particular switching strategy we adopt. Finally we 
need to add the contributions from $v_r^{*, \ell + 1}$'s and the extra crossings that we make in 
order to link $\pi_{i, 1}^{*, \ell}$ and $\pi_{i, 2}^{*, \ell}$ for $i \in [2]$.\par

Now we carry out the detailed computations. It is expected that $d_{\gamma, \ell, j}$ is approximately uniform over 
$j\in [\Gamma]$. However, proving this requires some effort. It turns out easier to treat the ``hypothetical'' case that $d_{\gamma, \ell}$ is dominated by a small 
number of $d_{\gamma, \ell, j}$'s. We also include a case to deal with small values of $\ell$ when $c_\ell^2$ is not big enough for our ``main'' strategy to work.

\noindent {\bf Case 1:} $\sum_{j=1}^\Gamma d_{\gamma, \ell, j} \mathbf 1_{\{d_{\gamma, \ell, j} 
\geq \Gamma^{-2/3} d_{\gamma, \ell}\}} \geq d_{\gamma, \ell} \Gamma^{-1/10}$. 

\vspace*{0.15cm}

It is easy to see that 
\begin{equation}\label{eq-var-lower-bound}
\sum_{j=1}^\Gamma d_{\gamma, \ell, j}^2 \geq d_{\gamma, \ell}^2 \Gamma^{-13/15}\,.
\end{equation}
In this case our switching strategy is simple. For each $k = 1, 2$, we choose the rectangle that 
minimizes $\sum_{j=1}^\Gamma d_{\gamma, \ell, j} Z_{\ell; i, k, j}$. Recall from \eqref{eq:Z_def} that
$$Z_{\ell; i, k, j} = (-1)^{(k - 1)\Gamma + j}b_\ell a_{V_{2L; \lceil((k-1)\Gamma + j / 2) \rceil}} + (-1)^i c_\ell a_{R_{\ell; k, j}}\,.$$
Since $a_{V_{2L; \lceil ((k-1)\Gamma + j / 2) \rceil}}$'s are centered, we have 
$$\E \Big(\min_{i = 1, 2}\sum_{j=1}^\Gamma d_{\gamma, \ell, j}Z_{\ell; i, k, j}\Big) = -c_\ell \E |\sum_{j=1}^\Gamma d_{\gamma, \ell, j}a_{R_{\ell; i, k, 
j}}|\,.$$ 
But $\sum_{j=1}^\Gamma d_{\gamma, \ell, j}a_{R_{\ell; i, k, j}}$ is a centered Gaussian variable with variance 
$\sum_{j = 1}^\Gamma d_{\gamma, \ell, j}^2$. Thus from \eqref{eq-var-lower-bound} we get that for $k \in [2]$,
\begin{align*}
\E \tilde{D}_{\gamma, \ell; k} \leq d_{\gamma, \ell}(1 + \gamma^2/ 2 + O(\gamma^3)) + \E 
\min_{i=1, 2} \sum_{j=1}^\Gamma d_{\gamma, \ell, j} \gamma Z_{\ell; i, k, j} \\ \leq 
d_{\gamma, \ell} (1 + O(\gamma^2) - \Omega(\gamma)\Gamma^{-13/30})\leq d_{\gamma, \ell} (1 - 
\Omega(\gamma^{1.9}))\,,
\end{align*}
where the last inequality follows from our assumption that $\Gamma \geq \alpha / \gamma^2$. Let $u_k \in A_{\ell, \Gamma}$ be such that $V_L^{\Gamma, u_k}$ is 
the minimizer. In order to connect $\pi^{*, u_1, \ell}$ and $\pi^{*, u_2, \ell}$ into a crossing of 
$V_{2L}^{\Gamma}$, we need some additional paths. To this end we add two vertical crossings through the rectangles $V_{\mathrm{left}} = \big([(\Gamma-1) L , \Gamma L - 1] \times [0, 2L-1]\big) \cap \Z^2$ and $V_{\mathrm{right}} = \big([\Gamma L , (\Gamma + 1)L - 
1] \times [0, 2L-1]\big) \cap \Z^2$. We also add two horizontal crossings through the rectangles $V_{\mathrm{up}} = \big([(\Gamma - 1)L, (\Gamma + 1) L - 1] \times [L, 2L-1]\big) \cap \Z^2$ and $V_{\mathrm{down}} = \big([(\Gamma - 1)L, (\Gamma + 1) L 
- 1] \times [0, L-1]\big) \cap \Z^2$. All these 
crossings have minimum weights. We can now form $\pi^{*, \ell+1}$ by concatenating $\pi^{*, u_1, \ell}, \pi^{*, 
u_2, \ell}$ and these four crossings. See 
Figure~\ref{fig:Case1} for an illustration. It is obvious from the construction that $\pi^{*, \ell + 1}_{L, coarse}$ is simple, symmetric in law with respect to the line $y = L - 0.5$ and entirely determined by the random variables $\{a_{R_{\ell; k, j}}\}_{k \in [2], 
j \in [\Gamma]}$. Thus the hypotheses~\ref{induct:hypo2} 
and \ref{induct:hypo3} hold. The only possible case where the hypothesis \ref{induct:hypo4} could fail would be the line $\mathbb L_{\Gamma L}$. But we do not need to 
consider this line as $\Gamma$ is odd. Finally notice that, since $\Gamma \gg 1$ and minimum crossing weights through adjacent rectangles are super-additive, each of the four additional crossings will have expected weight 
bounded by, say, $\frac{2.1}{\Gamma}d_{\gamma, \ell}$. 
Using the fact that $\alpha \gg 1$, we then deduce that in this case
\begin{equation*}
d_{\gamma, \ell + 1} \leq (2 - \gamma^2/2) d_{\gamma, \ell} \leq 2\Gamma(2 - 0.2)^{\ell + 1}\,.
\end{equation*}

\begin{figure}[!htb]
 \centering
\begin{tikzpicture}[semithick, scale = 0.8]
\draw [dashed] (0, -1.9) -- (0, 1.9);
\draw [dashed] (-5.9, 0) -- (5.9, 0);
\draw (0.1, 0.1) rectangle (1.9, 1.9);

\draw (2.1, 0.1) rectangle (3.9, 1.9);

\draw (4.1, 0.1) rectangle (5.9, 1.9);

\draw (-0.1, 0.1) rectangle (-1.9, 1.9);

\draw (-2.1, 0.1) rectangle (-3.9, 1.9);

\draw (-4.1, 0.1) rectangle (-5.9, 1.9);

\draw (0.1, -0.1) rectangle (1.9, -1.9);

\draw (2.1, -0.1) rectangle (3.9, -1.9);

\draw (4.1, -0.1) rectangle (5.9, -1.9);

\draw (-0.1, -0.1) rectangle (-1.9, -1.9);

\draw (-2.1, -0.1) rectangle (-3.9, -1.9); 

\draw (-4.1, -0.1) rectangle (-5.9, -1.9);

\draw [red, style = {decorate, decoration = {snake, amplitude = 0.6}}] (-5.9, 1) -- (-4.1, 1);

\draw [red, style = {decorate, decoration = {snake, amplitude = 0.6}}] (-3.9, 1) -- (-2.1, 1);

\draw [red, style = {decorate, decoration = {snake, amplitude = 0.6}}] (-1.9, 1) -- (-1, 1);

\draw [style = {decorate, decoration = {snake, amplitude = 0.6}}] (-1, 1) -- (-0.1, 1);

\draw [style = {decorate, decoration = {snake, amplitude = 0.6}}] (-1, 1.9) -- (-1, 1);
\draw [red, style = {decorate, decoration = {snake, amplitude = 0.6}}] (-1, 1) -- (-1, 0.1);
\draw [red, style = {decorate, decoration = {snake, amplitude = 0.6}}] (-1, -0.1) -- (-1, -0.55);
\draw [style = {decorate, decoration = {snake, amplitude = 0.6}}] (-1, -0.55) -- (-1, -1.9);
\draw [style = {decorate, decoration = {snake, amplitude = 0.6}}] (-1.9, -0.55) -- (-1, -0.55);
\draw [red, style = {decorate, decoration = {snake, amplitude = 0.6}}] (-1, -0.55) -- (-0.1, -0.55);
\draw [red, style = {decorate, decoration = {snake, amplitude = 0.6}}] (0.1, -0.55) -- (1, -0.55);
\draw [style = {decorate, decoration = {snake, amplitude = 0.6}}] (1, -0.55) -- (1.9, -0.55);
\draw [red, style = {decorate, decoration = {snake, amplitude = 0.6}}] (1, -0.55) -- (1, -1);
\draw [style = {decorate, decoration = {snake, amplitude = 0.6}}] (1, -0.55) -- (1, -0.1);
\draw [style = {decorate, decoration = {snake, amplitude = 0.6}}] (1, -1) -- (1, -1.9);

\draw [red, style = {decorate, decoration = {snake, amplitude = 0.6}}] (5.9, -1) -- (4.1, -1);

\draw [red, style = {decorate, decoration = {snake, amplitude = 0.6}}] (3.9, -1) -- (2.1, -1);

\draw [red, style = {decorate, decoration = {snake, amplitude = 0.6}}] (1.9, -1) -- (1, -1);

\draw [style = {decorate, decoration = {snake, amplitude = 0.6}}] (1, -1) -- (0.1, -1);

\end{tikzpicture}
 \caption{{\bf The construction of $\pi^{*, \ell+1}$ for Case 1.} Here $\Gamma = 3$. 
The segments of $\pi^{*, \ell+1}$ have been colored in red.}
 \label{fig:Case1}
\end{figure}
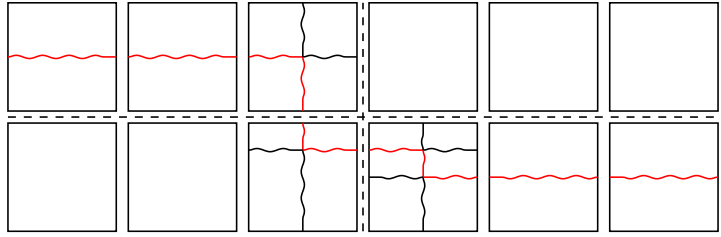
\noindent {\bf Case~2:} $\ell < 60$. 

\vspace*{0.15cm}

Here we use the strategy described in Case~1. It then follows from the previous discussions that
\begin{equation*}
d_{\gamma, \ell + 1} \leq (2 + 2\gamma^2)d_{\gamma, \ell}\,.
\end{equation*}
Since $d_{\gamma, 0} = \Gamma$, we get by repeated application of the previous inequality 
\begin{equation*}
d_{\gamma, \ell + 1} \leq \Gamma(2 + 2\gamma^2)^{\ell + 1} \leq 2\Gamma(2 - 0.2\gamma^2)^{\ell + 1}\,.
\end{equation*}
\noindent {\bf Case~3.} $\sum_{j=1}^\Gamma d_{\gamma, \ell, j} \mathbf 1_{\{d_{\gamma, \ell, j} \geq 
\Gamma^{-2/3} d_{\gamma, \ell}\}} \leq d_{\gamma, \ell} \Gamma^{-1/10}$ and $\ell \geq 60$. 

\vspace*{0.15cm}

This is the main and real case. We will also assume that 
$$d_{\gamma, \ell} \geq 2\Gamma (2 - 0.2\gamma^2)^{\ell} (1 - \gamma^2)\,,$$
for otherwise we can simply follow the strategy in Case~1. In this case our switching strategy 
will in fact be informed by the way we plan to contruct the $\mathcal L$-segments $v_r^{*, \ell + 
1}$'s. So we first give a detailed description of the latter. 
We refer the reader to Fig~\ref{fig:vertical_embed} for an illustration of the construction.\\
For each $r \leq R$, let $V_{L;{U(r)}}$ and $V_{L; D(r)}$ respectively denote the $L \times L$ 
boxes in the top and bottom row of $V_{2L}^\Gamma$ containing the endpoints of $H_r$ and $H_{r + 
1}$. Call the $L \times 2L$ rectangle formed by $V_{L; U(r)}$ and $V_{L; D(r)}$ as $V_{L; r}$. 
Suppose the right boundary of $V_{L; r}$ lies along the line $\mathbb{L}_{j_rL}$ where $j_r \in 
[2\Gamma - 1]$. It suffices to consider the case when $j_r \leq \Gamma$. Let $v_{j_r, L, L_{100}}^u$ and $v_{j_r, L, L_{100}}^d$ respectively denote the last points along $\pi^{*, (0, L), 
\ell}_{L_{100}, coarse}$ and $\pi^{*,\ell}_{L_{100}, coarse}$ that lie to the left of $\mathbb 
L_{j_rL}$ where $L_{100} = \max(2^{-100}L, 1)$. Now if $L_{100} = 1$, we simply set $v_r^{*, \ell + 1}$ to be the (discrete) straight line joining $v_{j_r, L, 
L_{100}}^u$ and $v_{j_r, L, L_{100}}^d$. Otherwise $v_{j_r, L, L_{100}}^u$ and $v_{j_r, L, L_{100}}^d$ are the centers of two $L_{100} \times L_{100}$ boxes, say, 
$A_{j_r, L}^u$ and $A_{j_r, L}^d$ respectively. Let $\tilde{A}_{j_r, L}^u$ (respectively $\tilde{A}_{j_r, L}^d$) be the box around $A_{j_r, L}^u$ (respectively 
$\tilde{A}_{j_r, L}^d$) of side length $2L_{100}$. Then we construct light contours in each of the annuli $\tilde{A}_{j_r, L}^u \setminus A_{j_r, L}^u$ and 
$\tilde{A}_{j_r, L}^d \setminus A_{j_r, L}^d$. These can easily be achieved by concatenating four minimum weight 
crossings in four rectangles. In addition, we let $\tilde A_r$ be a rectangle whose left top corner is $v_{j_r, L, L_{100}}^u + (-0.5, L_{100} - 0.5)$ and right bottom corner is $v_{j_r, L, L_{100}}^d + (L_{100}/2 - 0.5, -L_{100} 
+ 0.5)$. We construct a minimum weight 
vertical crossing through $\tilde A_r$. Finally we concatenate these light crossings into $v_r^{*, \ell 
+ 1}$. The total expected weight of these crossings can 
be bounded by our induction hypotheses. Notice that we can always concatenate these paths in a way so that our construction of $v_r^{*, \ell + 1}$ obeys induction hypothesis~\ref{induct:hypo4}.
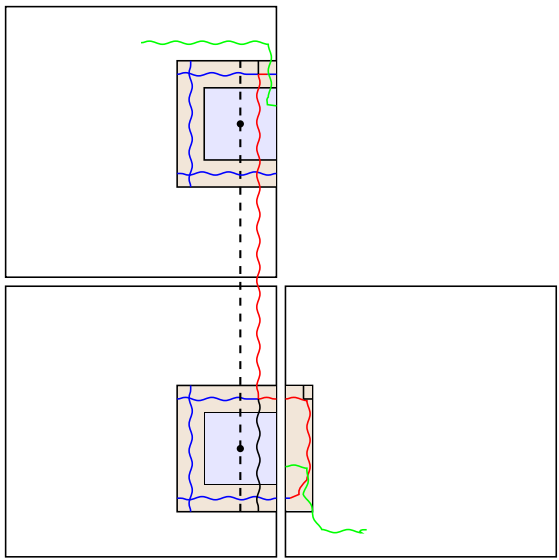
\begin{figure}[!htb]
 \centering
 \begin{tikzpicture}[semithick, scale = 1.2]
 \draw (-3.1, 0) rectangle (-0.1, 3);
 \draw (0, 0) rectangle (3, 3);
 \draw (-3.1, 3.1) rectangle (-0.1, 6.1);
 
 \draw (-0.1, 0.8) -- (-0.9, 0.8) -- (-0.9, 1.6) -- (-0.1, 1.6);
 
 \fill [blue!10] (-0.11, 0.81) rectangle (-0.89, 1.59); 
 \fill [blue!10] (-0.11, 4.4) rectangle (-0.9, 5.19);
 
 \fill [brown!20](-0.11, 1.61) rectangle (-1.2, 1.89);
 \fill [brown!20](-0.11, 0.5) rectangle (-1.2, 0.79);
 \fill [brown!20](-0.91, 0.79) rectangle (-1.2, 1.89);
 
 \fill [brown!20](-0.11, 5.2) rectangle (-1.2, 5.5);
 \fill [brown!20](-0.11, 4.1) rectangle (-1.2, 4.39);
 \fill [brown!20](-0.91, 4.39) rectangle (-1.2, 5.49);
 
 \draw (-0.1, 0.5) -- (-1.2, 0.5) -- (-1.2, 1.9) -- (-0.1, 1.9);
 \draw (0, 0.5) -- (0.3, 0.5) -- (0.3, 1.9) -- (0, 1.9);
 \fill [brown!20](0.29, 0.52) rectangle (0.01, 1.89);
 
 \draw [blue, style = {decorate, decoration = {snake , amplitude = 0.6}}] (-0.1, 0.65) -- (-1.2, 0.65);
 \draw [blue, style = {decorate, decoration = {snake , amplitude = 0.6}}] (-1.05, 0.5) -- (-1.05, 1.9);
 \draw [blue, style = {decorate, decoration = {snake , amplitude = 0.6}}] (-1.2, 1.75) -- (-0.3, 1.75);
 \draw [red, style = {decorate, decoration = {snake , amplitude = 0.6}}] (-0.3, 1.75) -- (-0.1, 1.75);
 \draw [red, style = {decorate, decoration = {snake , amplitude = 0.6}}] (0, 1.75) -- (0.15, 1.75) -- (0.15, 0.65) -- (0.05, 0.65);	
 \draw [blue] (0, 0.65) -- (0.05, 0.65);
 \draw [green, style = {decorate, decoration = {snake , amplitude = 0.6}}] (0, 1) -- (0.1, 1) -- (0.05, 0.65) -- (0.2, 0.3) -- (0.9, 0.3);
 \draw [style = {decorate, decoration = {snake , amplitude = 0.4}}] (0.2, 1.75) -- (0.3, 1.75);
 \draw [style = {decorate, decoration = {snake , amplitude = 0.4}}] (0.2, 1.9) -- (0.2, 1.75);
 
 \fill (-0.5, 1.2) circle [radius = 0.04];
 \draw [thick, dashed] (-0.5, 0.5) -- (-0.5, 5.5);
 \fill (-0.5, 4.8) circle [radius = 0.04];

 \draw (-0.1, 4.4) -- (-0.9, 4.4) -- (-0.9, 5.2) -- (-0.1, 5.2);
 \draw (-0.1, 4.1) -- (-1.2, 4.1) -- (-1.2, 5.5) -- (-0.1, 5.5);
  \draw [blue, style = {decorate, decoration = {snake , amplitude = 0.6}}] (-0.1, 4.25) -- (-1.2, 4.25);
  \draw [blue, style = {decorate, decoration = {snake , amplitude = 0.6}}] (-1.05, 4.1) -- (-1.05, 5.5);
  \draw [blue, style = {decorate, decoration = {snake , amplitude = 0.6}}] (-1.2, 5.35) -- (-0.3, 5.35);
  \draw [red, style = {decorate, decoration = {snake , amplitude = 0.6}}] (-0.3, 5.35) -- (-0.2, 5.35);
  \draw [blue, style = {decorate, decoration = {snake , amplitude = 0.6}}] (-0.2, 5.35) -- (-0.1, 5.35);
  \draw [green, style = {decorate, decoration = {snake , amplitude = 0.6}}] (-1.6, 5.7) -- (-0.2, 5.7) -- (-0.2, 5) -- (-0.1, 5);
  \draw [style = {decorate, decoration = {snake , amplitude = 0.6}}] (-0.3, 5.5) -- (-0.3, 5.35);

  \draw [red, style = {decorate, decoration = {snake , amplitude = 0.6}}] (-0.3, 5.35) -- (-0.3, 1.75);
  \draw [style = {decorate, decoration = {snake , amplitude = 0.6}}] (-0.3, 1.75) -- (-0.3, 0.5);

 \end{tikzpicture}
\caption{{\bf The construction of $v_r^{*, \ell + 1}$.} The two big boxes on the left define the rectangle $V_{L; r}$, the top and bottom one being $V_{L;U(r)}$ and $V_{L;D(r)}$ respectively. The two boxes filled with light blue are $A_{j_r, L}^u$ (top) and $A_{j_r, L}^d$ (bottom). The annuli $\tilde A_{j_r, L}^u \setminus A_{j_r, L}^u$ and $\tilde A_{j_r, L}^d \setminus A_{j_r, L}^d$ (or portions thereof) are indicated by light brown.
The path colored in red is $v_r^{*, \ell + 1}$. The green paths represent $h_r^{*, \ell + 1}$ and $h_{r+1}^{*, \ell + 1}$.}
\label{fig:vertical_embed}
\end{figure}

Let us now try to bound the expected weight of $v_r^{*, \ell + 1}$ where we assign vertex weight 
based on the variables \emph{associated with boxes of side length $ \leq L$}. Notice that this does 
not require the knowledge of the particular switching strategy we employ. Denote by $D_r$ 
the total weight of the light crossings in the two annuli where we assign vertex weight based 
on variables associated with the boxes of side length $\leq \tfrac{L_{100}}{2} \vee 1$. By iterating the induction hypotheses~\ref{induct:hypo4} and \ref{induct:hypo2} (see also the argument presented in the proof of Lemma~\ref{lem:uniform}) we get that the $L_{100}$ coarsenings of $\pi^{*, u, \ell}$'s are independent of all the Gaussians at level lower than $L_{100}$, which implies
\begin{equation}\label{eq-weight-D-i}
\E D_{r} \leq \frac{33 d_{\gamma, (\ell - \log_2(2L/L_{100}))^+}}{\Gamma}\,, 
\end{equation}
where $x^+ = \max(x, 0)$. Let $\tilde D_r$ denote the weight of the vertical crossing through $\tilde{A}_r$ or the weight of $v_r^{*, \ell + 1}$ when $L_{100} = 1$ (where again, we assign the vertex weight based on Gaussian variables associated with boxes of side length 
$\leq \tfrac{L_{100}}{2} \vee 1$). Clearly we have
\begin{equation}\label{eq-weight-tilde-D-i}
\E \tilde{D}_r \leq \frac{\big( 3L_{100} + \E |v_{j_r, L, L_{100}}^u - v_{j_r, L, L_{100}}^d| 
\big)}{(L_{100} / 2)\Gamma}d_{\gamma, (\ell - \log_2(2L/L_{100}))^+}\,.
\end{equation}
In both \eqref{eq-weight-D-i} and \eqref{eq-weight-tilde-D-i}, we have silently used the fact that 
expected minimum weights of crossings through $V_{L'}^m$'s form a superadditive sequence in $m$. The following lemma says that we do not lose much in terms of the expected costs of these 
crossings after we take into account the contributions from Gaussians appearing in higher levels.
\begin{lemma}\label{lem-max-exponential-Gaussian}
Let $X_1, \ldots, X_m$ i.i.d.\ standard Gaussian variables, where $m \leq (1/\gamma)^{100}$ and 
$U$ be a nonnegative random variable bounded by $C$. Then 
$$\E (U\max_i \mathrm{e}^{\gamma X_i}) \leq  \E U \big(1 + O(\gamma \log (1/\gamma))\big) + 
C \mathrm{e}^{-\Omega(\log(1/\gamma))}\,.$$
\end{lemma}
\begin{proof}
Let $M = \max_i X_i$. It is clear that $\E M \leq \sqrt{200 \log (1/\gamma)}$, and also that 
$\P(M\geq \E M+ u)\leq e^{-u}$. The conclusion follows by a straightforward computation.
\end{proof}
\noindent Combining \eqref{eq-weight-D-i}, \eqref{eq-weight-tilde-D-i}, Lemmas~\ref{lem:uniform} 
and \ref{lem-max-exponential-Gaussian}, we get that the expected weight for each ``switching gadget''is bounded by $(1 + 10\times2^{-100})(2 - 0.2\gamma^2)^{\ell} \leq \frac{1 + 20\times 2^{-100}}{\Gamma}d_{\gamma, \ell}$, where we use the fact that $\gamma \ll 1$. Recall the expression for $\E\tilde{D}_{\gamma, \ell; k}$ as 
given in \eqref{eq:weight_upper_bnd} and the fact that 
$$Z_{\ell; i, k, j} = (-1)^{(k - 1)\Gamma + j}b_\ell a_{V_{2L; \lceil(k-1)\Gamma + 1 / 2 \rceil}} + (-1)^i c_\ell a_{R_{\ell; k, j}}\,.$$
It is then natural to consider switching strategies $\{i_{j, k}\}_{j \in [\Gamma], k \in [2]}$ that give a small value of the following quantity:
\begin{equation}
\label{eq:optimize}
c_\ell\gamma\sum_{j = 1}^\Gamma d_{\gamma, \ell, j}(-1)^{i_{j, k}} a_{R_{\ell; k, j}} + R\frac{1 + 20\times 2^{-100}}{\Gamma}d_{\gamma, \ell}\,.
\end{equation}
To this end we consider a similar optimization problem for a process that is related to this 
quantity. Denote by $J = \{j\in [\Gamma]: d_{\gamma, \ell, j} \leq \Gamma^{-2/3} d_{\gamma, 
\ell}\}$. By Cauchy-Schwartz inequality, we see that 
\begin{equation}\label{eq-var-lower-J}
\sum_{j\in J} d_{\gamma, \ell, j}^2 \geq \frac{(1 - \Gamma^{-1/5}) d_{\gamma, \ell}^2}{\Gamma}\,.
\end{equation}
Write $J = \{j_1, \ldots, j_{|J|}\}$ with $1 \leq j_1 < j_2 < \cdots < j_{|J|} \leq \Gamma$, and 
define the sequence
$$g_{\gamma, i} =c_\ell^2\gamma^2\mbox{$\sum_{k \leq i}$}d_{\gamma, \ell, j_k}^2 \mbox{ for } 1 
\leq i \leq 2|J|\,,$$
where $j_i = \Gamma + j_{i - |J|}$ and $d_{\gamma, \ell, j_i} = d_{\gamma, \ell, j_{i - |J|} }$ for $i > |J|$.
Consider the process
$$S_{t; \ell} = c_\ell \gamma \mbox{$\sum_{0 \leq k \leq i}$}  d_{\gamma, \ell, j_{k+1}} (W_{t_i 
\wedge j_{k+1}}^{\ell} - W_{t_i \wedge (j_{k+1} - 1)}^{\ell}) \mbox{ for } g_{\gamma, i} \leq t 
\leq g_{\gamma, i+1}\,,$$
where $t_i = j_{i+1} - 1 + \tfrac{t - g_{\gamma, i}}{c_\ell^2 \gamma^2 d_{\gamma, \ell, j_{i+1}}^2}$ and 
$g_{\gamma, 0} = 0$. Thus $\{S_{t; \ell}\}_{0 \leq t \leq g_{\gamma, 2|J|}}$ is a standard Brownian motion whose increment between the time points $g_{\gamma, i-1}$ and $g_{\gamma, i}$ is given by $c_\ell \gamma d_{\gamma, \ell, j_i} a_{R_{\ell; k, j_i}}$ 
for $1 \leq i \leq 2|J|$. Now suppose that $[T_0, T_1, \cdots, T_{R'+1}]$ is the partition of $[0, g_{\gamma, 2|J|}]$ given by Theorem~\ref{thm:brownian_switch} when applied to $\{S_{t; \ell}\}_{0 \leq t \leq g_{\gamma, 2|J|}}$ for $\lambda = \tfrac{(1+20\times 2^{-100})d_{\gamma, \ell}}{\Gamma}$ (with appropriate change of scales). 
Using this we will define $L$-coarsening of the path that we will construct. Since this has to be 
simple by our induction hypothesis, it suffices to assign an element in $\{-1, 1\}^{2|J|}$ to 
$j_1, j_2, \cdots, j_{2|J|}$. For this purpose define the random variables $\mathrm{sgn}_{1; 
\ell}, \cdots, \mathrm{sgn}_{2|J|; \ell}$ as follows:
\begin{equation*}
    \mathrm{sgn}_{i; \ell} = \begin{cases}
               -1 & \text{if }T_r < g_{\gamma, i} \leq T_{r + 1} \mbox{ such that } S_{T_r; 
               \ell} <  
               S_{T_{r+1}; \ell}\\
                1 & \text{if }T_r < g_{\gamma, i} \leq T_{r + 1} \mbox{ such that } S_{T_r; 
                \ell} >  
               S_{T_{r+1}; \ell}\,.
           \end{cases}
\end{equation*} 
It is clear that $(\mathrm{sgn}_{1; \ell}, \cdots, \mathrm{sgn}_{2|J|; \ell})$ is identically 
distributed as $(-\mathrm{sgn}_{1; \ell}, \cdots, -\mathrm{sgn}_{2|J|; \ell})$. As a consequence 
the distribution of $\pi^{*, \ell + 1}_{L, coarse}$ is invariant relative to reflection about the 
line $y = L - 0.5$. \par

Now we collect the contributions to the weight of $\pi^{*, \ell+1}$ from all the sources. We use 
crossings through $V_{\textrm{up}}, V_{\textrm{down}}, V_{\textrm{left}}$ and $V_{\textrm{right}}$, as in Case 1, to link the crossings $\pi^{*, \ell}_{i, 1}$ and 
$\pi^{*, \ell}_{i, 1}$ for $i \in [2]$. As discussed in Case 1, this will not tamper with induction 
hypothesis~\ref{induct:hypo4}. We also need a bound on the error we made by interpolating $(S_{t; \ell})$ as a 
Brownian motion. Notice that for each $(T_r, T_{r + 1})$ that is contained in some interval $(g_{\gamma, i-1}, g_{\gamma, i})$, we lose at most $M_{\mathrm{dis}}$ due to ``discretization'' of the optimal partition obtained from Theorem~\ref{thm:brownian_switch}, where 
$$M_{\mathrm{dis}} = \max_{1 \leq i \leq 2|J|}\max_{g_{\gamma, i-1} \leq s, t \leq g_{\gamma, i}} |S_{t; \ell} - S_{s; \ell}|\,.$$
Thus the total loss we incur due to discretization is at most $M_{\mathrm{dis}}R'$. We use a 
simple estimate for $M_{\mathrm{dis}}$ as follows:
\begin{equation}\label{eq-discretization}
\E M_{\mathrm{dis}}  \leq O(\gamma \Gamma^{-2/3} \log(\Gamma)) d_{\gamma, \ell}, \mbox{ and } \P(|M_{\mathrm{dis}} - \E M_{\mathrm{dis}}| \geq u\gamma \Gamma^{-2/3} \log(\Gamma)d_{\gamma, \ell} ) \leq e^{-\Omega(u^2)}\,,
\end{equation}
where we use the fact that $d_{\gamma, \ell, j} \leq \Gamma^{-2/3}d_{\gamma, \ell}$ on $J$.
Finally we need to make amends for the fact that what we pass as ``penalty'' term to \eqref{eq:optimize} is 
actually based on the Gaussians up to level $\ell$. But the multiplier at level $\ell + 1$ can be at most $\mathrm{e}^{M_{\ell + 1}}$, where $M_{\ell+1}$ 
is the maximum of all Gaussians added at level $\ell + 1$. Altogether, we deduce that
\begin{align}\label{eq-key-induction}
d_{\gamma, \ell+1} &\leq  2 d_{\gamma, \ell} (1 + \gamma^2/2 + O(\gamma^3) + 
\frac{10}{\Gamma}) - \mathcal I_\ell + \sqrt{\E R'^2}\sqrt{\E M_{\mathrm{dis}}^2} \nonumber \\ 
& + \frac{(1+20\times 2^{-100})d_{\gamma, \ell}}{\Gamma} \sqrt{\E R'^2}\sqrt{\E(\e^{\gamma M_{\ell+1}}- 1)^2}\,,
\end{align}
where 
\begin{equation*}
\mathcal I_{\ell} = \E \bigg( \sum_{i=1}^{R'+1} |S_{T_i; \ell} - S_{T_{i-1}; \ell}| - R' \frac{(1+20\times 2^{-100})d_{\gamma, \ell}}{\Gamma}\bigg)\,,
\end{equation*}
It follows from condition~(d) in Theorem~\ref{thm:brownian_switch} and a straightforward computation that
\begin{equation}
\label{eq:switch_bound}
\E R'^2 \leq O\bigg(\Big(\frac{\Gamma^2\gamma^2\sum_{j \in J}d_{\gamma, \ell, j}^2}{d_{\gamma, \ell}^2}\Big)^2\bigg)\,.
\end{equation}
In addition, by Theorem~\ref{thm:brownian_switch}, we deduce that for all $\alpha > 0$ 
$$\mathcal I_\ell \geq \big(\frac{4}{3} - 2^{-100}\big) \gamma^2 \Gamma \frac{\sum_{j\in J} d_{\gamma, \ell, 
j}^2}{d_{\gamma, \ell}}\,.$$
Here we used the fact that $\ell \geq 60$. Plugging the bounds from \eqref{eq-discretization}, \eqref{eq:switch_bound} and Lemma~\ref{lem-max-exponential-Gaussian} into \eqref{eq-key-induction} we obtain
$$d_{\gamma, \ell+1} \leq (2 - 0.6\gamma^2)2\Gamma(2 - 0.2\gamma^2)^\ell \leq 2\Gamma(2 - 
0.2\gamma^2)^{\ell + 1}\,.$$
This completes the proof of Theorem~\ref{thm-main}.

\end{document}